\documentclass[11pt]{article}
\usepackage{amsthm}
\usepackage{amsmath}
 \usepackage{a4}                          
\usepackage{amssymb}
\DeclareMathOperator{\Log}{Log}

\DeclareMathOperator{\Rea}{Re}
 
\newcommand{\Int}{\int_{-\infty}^\infty}
 \newcommand{\eps}{\varepsilon}
 
 \renewcommand{\a}{\alpha}
 \newcommand{\La}{\Lambda}

\newcommand{\C}{\mathbb{C}}
\newcommand{\R}{\mathbb{R}}
\newcommand{\N}{\mathbb{N}}

\newtheorem{thm}{Theorem}[section]
\newtheorem{lemma}[thm]{Lemma}
\newtheorem{prop}[thm]{Proposition}
\newtheorem{rem}[thm]{Remark}

\title{A two-parameter extension of the Urbanik semigroup}

\author{Christian Berg}
\begin{document}
\maketitle

\begin{abstract} We prove that $s_n(a,b)=\Gamma(an+b)/\Gamma(b), n=0,1,\ldots$ is an infinitely divisible Stieltjes moment sequence for arbitrary $a,b>0$. Its powers $s_n(a,b)^c, c>0$ are Stieltjes determinate if and only if $ac\le 2$. The latter was conjectured in a paper by Lin (ArXiv: 1711.01536) in the case $b=1$. We describe a product convolution semigroup $\tau_c(a,b),\, c>0$  
of probability measures on the positive half-line with densities $e_c(a,b)$  and having the moments
$s_n(a,b)^c$. We determine the asymptotic behaviour of $e_c(a,b)(t)$ for $t\to 0$ and for $t\to\infty$, and the latter implies the Stieltjes indeterminacy when $ac>2$. The results extend previous work of the author and J. L. L\'opez and lead to a convolution semigroup of probability densities $(g_c(a,b)(x))_{c>0}$ on the real line. The special case $(g_c(a,1)(x))_{c>0}$
are the convolution roots of the Gumbel distribution with scale parameter $a>0$. All the densities
$g_c(a,b)(x)$ lead to determinate Hamburger moment problems.  
  \end{abstract}


\section{Introduction} A Stieltjes moment sequence is a sequence of non-negative numbers of the form
\begin{equation}\label{eq:St}
s_n=\int_0^\infty t^n\,d\mu(t),\quad n\in\N_0:=\{0,1,2,\ldots\},
\end{equation}
where $\mu$ is a positive measure on $[0,\infty)$ such that $x^n\in L^1(\mu)$ for all $n\in\N_0$.
The sequence $(s_n)$ is called normalized if $s_0=\mu([0,\infty))=1$, and it is called S-determinate 
(resp. S-indeterminate) if \eqref{eq:St} has exactly one (resp. several) solutions $\mu$ as positive measures on $[0,\infty)$. All these concepts go back to the fundamental memoir of Stieltjes \cite{St}.  

A Stieltjes moment sequence $(s_n)$ is called infinitely divisible if $(s_n^c)$ is a Stieltjes moment sequence for any $c>0$. These sequences were characterized in Tyan's phd-thesis \cite{Ty} and again in \cite{B2} without the knowledge of \cite{Ty}.
An important example of an infinitely divisible normalized Stieltjes moment sequence is $s_n=n!$, first established in Urbanik \cite{U1}. He proved that $e_c$ in \eqref{eq:Urb1} is a probability density such that
\begin{equation}\label{eq:Urb1}
(n!)^c=\int_0^\infty t^n e_c(t)\,dt,\quad e_c(t)=\frac{1}{2\pi}\Int t^{ix-1}\Gamma(1-ix)^c\,dx,\quad c,t >0.
\end{equation} 
Here $\Gamma$ is Euler's Gamma-function. The family $\tau_c=e_c(t)dt,c>0$ is a convolution semigroup in the sense of \cite{B:F} on the  locally compact abelian group $G=(0,\infty)$ under multiplication. It is called the Urbanik semigroup in \cite{B:L}.

By Carleman's criterion for S-determinacy it is easy to prove that $(n!)^c$ is S-determinate for $c\le 2$. That this estimate is sharp was first proved in \cite{B1}, where it was established that $(n!)^c$ is S-indeterminate for $c>2$ based on asymptotic results of Skorokhod \cite{Sk} about stable distributions, see \cite{Z}.  Another proof of the S-indeterminay was given in \cite{B:L} based on the asymptotic behaviour of $e_c(t)$,
\begin{equation}\label{eq:BL1}
e_c(t) = \frac{(2\pi)^{(c-1)/2}}{\sqrt{c}}\frac{\exp(-ct^{1/c})}{t^{(c-1)/(2c)}}
\left[1+\mathcal O\left(t^{-1/c}\right)\right],
\quad t\to\infty. 
\end{equation}

In the recent paper \cite{Lin}, Lin proposes the following conjecture:

{\bf Conjecture} {\it Let $a>0$ be a real constant and let $s_n=\Gamma(na+1), n\in\N_0$. Then

(a) $(s_n)$ is an infinitely divisible Stieltjes moment sequence;

(b) For real $c>0$ the sequence $(s_n^c)$ is S-determinate if and only if $ac\le 2$;

(c) For $0<c\le 2/a$ the unique probability measure $\mu_c$ corresponding to $(s_n^c)$ has the Mellin transform
$$
\int_0^\infty t^s\,d\mu_c(t)=\Gamma(as+1)^c,\quad s\ge 0.
$$
}

When $a=1$ the conjecture is true because of the known results about the Urbanik semigroup, and for $a\in \N,\,  a\ge 2$ the conjecture is true because of the Theorems 5 and 8 in \cite{Lin}.

We shall prove that the conjecture is true, and it is a special case of similar results for the following more general normalized Stieltjes moment sequence
 
\begin{equation}\label{eq:s(a,b)}
s_n(a,b)=\frac{\Gamma(an+b)}{\Gamma(b)}=\frac{1}{a\Gamma(b)}\int_0^\infty t^n t^{b/a -1}\exp(-t^{1/a})\,dt,\quad n=0,1,\ldots, 
\end{equation}
where $a,b>0$ are arbitrary.
 
Defining
\begin{equation}\label{eq:ta}
e_1(a,b)(t)=\frac{1}{a\Gamma(b)}t^{b/a -1}\exp(-t^{1/a}),
\end{equation} 
we get for $\Rea z>-b/a$ and a change of variable $t=s^{a}$
\begin{equation}\label{eq:mel}
\int_0^\infty t^z e_1(a,b)(t)\,dt=\Gamma(az+b)/\Gamma(b).
\end{equation}

This leads to our main result.

\begin{thm}\label{thm:Urbab} 
(i) $(s_n(a,b))$ is an infinitely divisible Stieltjes moment sequence. 

(ii) There exists a uniquely determined convolution semigroup $(\tau_c(a,b))_{c>0}$ of probability measures on the multiplicative group $(0,\infty)$ such that
\begin{equation}\label{eq:s(a,b)z}
\int_0^\infty t^z\,d\tau_c(a,b)(t)=[\Gamma(az+b)/\Gamma(b)]^c,\quad \Rea z>-b/a, 
\end{equation}
and in particular $(s_n(a,b)^c)$ is the moment sequence of $\tau_c(a,b)$.

(iii)  $\tau_c(a,b)=e_c(a,b)(t)\,dt$ on $(0,\infty)$, where
\begin{equation}\label{eq:ec(a,b)}
e_c(a,b)(t)=\frac{1}{2\pi}\Int t^{ix-1}[\Gamma(b-iax)/\Gamma(b)]^c\,dx,\quad t>0
\end{equation}
is a probability density belonging to $C^\infty(0,\infty)$.

(iv) $(s_n(a,b)^c)$ is S-determinate if and only if $ac\le 2$, hence independent of $b>0$.
\end{thm}

Note that \eqref{eq:s(a,b)} is a special case of \eqref{eq:mel}.

The measure $\tau_1(a,b)$ was considered in \cite{Ta}, where it was proved that the measure  is S-indeterminate if $a>\max(2,2b)$. This is a consequence of our result. Note that  $\tau_1(a,1)$ is called the Weibull distribution with shape parameter $1/a$ and scale parameter 1. 

In \eqref{eq:s(a,b)z} and \eqref{eq:ec(a,b)} we use that $\Gamma(z)$ is a non-vanishing holomorphic
function in the cut plane 
\begin{equation}\label{eq:cut}
\mathcal A=\mathbb C\setminus (-\infty,0],
\end{equation}
so we can define
$$
\Gamma(z)^c:=\exp(c\log\Gamma(z)),\quad z\in \mathcal A
$$
using the holomorphic branch of $\log\Gamma$ which is 0 for $z=1$. This branch is explicitly given in \eqref{eq:holbr}.

Let us recall a few facts about convolution semigroups of probability measures on  LCA-groups, see \cite{B:F} for details.

The continuous characters of the multiplicative group $G=(0,\infty)$  can be given as $t\to
t^{ix}$, where $x\in\R$ is arbitrary, and in this way the dual
group $\widehat{G}$ of $G$ can be identified with the additive group
of real numbers. The convolution between measures $\mu,\sigma$ on $(0,\infty)$, called product convolution and denoted $\mu\diamond\sigma$, is defined as
$$
\int_0^\infty f(t)\,d\mu\diamond\sigma(t)=\int_0^\infty\int_0^\infty f(ts)\,d\mu(t)\,d\sigma(s)
$$  
for suitable classes of continuous functions $f$ on $(0,\infty)$, e.g. those of compact support.

A family $(\mu_c)_{c>0}$ of probability measures on the multiplicative group $G=(0,\infty)$ is called a convolution semigroup,  if $\mu_c\diamond \mu_d=\mu_{c+d}, c,d>0$ and $\lim_{c\to 0}\mu_c=\varepsilon_1$ vaguely. Here $\varepsilon_1$ is the Dirac measure with total mass 1 concentrated in the neutral element 1 of the group. By \cite[Theorem 8.3]{B:F} there is a one-to-one correspondence between convolution semigroups $(\mu_c)_{c>0}$ of probability measures on $G$ and continuous negative definite functions $\rho:\R\to\C$ satisfying $\rho(0)=0$ such that
\begin{equation}\label{eq:conv}
\int_0^\infty t^{-ix}\,d\mu_c(t)=\exp(-c\rho(x)),\quad c>0,x\in\R.
\end{equation}

 By the inversion theorem of Fourier analysis for
LCA-groups, if $\exp(-c\rho)$ is integrable on $\R$, then $\mu_c=f_c(t)\,dt$ for a continuous function $f_c(t)$
($tf_c(t)$ is the density of $\mu_c$  with respect to Haar measure $(1/t)dt$ on $(0,\infty)$) given by
\begin{equation}\label{eq:inv}
f_c(t)=\frac{1}{2\pi}\Int t^{ix-1}\exp(-c\rho(x))\,dx,\quad t>0.
\end{equation} 
(Note that the dual Haar measure of $(1/t)dt$ on $(0,\infty)$ is $1/(2\pi)\,dx$ on $\R$.)
 
\begin{prop}\label{thm:negdef}
For $a,b>0$ 
\begin{equation}\label{eq:U5}
\rho(x):=\log\Gamma(b)-\log\Gamma(b-iax),\quad x\in \R
\end{equation}
 is a continuous negative definite function on $\R$  satisfying $\rho(0)=0$.
\end{prop}
 
Proposition~\ref{thm:negdef} shows that there exists a uniquely determined product convolution semigroup
$(\tau_c(a,b))_{c>0}$ satisfying 

\begin{eqnarray}\label{eq:ix}
\int_0^\infty t^{-ix}\,d\tau_c(a,b)(x)&=&\exp[-c(\log\Gamma(b)-\log\Gamma(b-iax))]\nonumber\\
&=&[\Gamma(b-iax)/\Gamma(b)]^c,\quad x\in\R.
\end{eqnarray}
Like in the proof of \cite[Lemma 2.1]{B1} it is easy to see that \eqref{eq:ix} implies \eqref{eq:s(a,b)z}.

Putting $z=-ix$ in \eqref{eq:mel}, we see by the uniqueness theorem for Fourier transforms that 
$\tau_1(a,b)=e_1(a,b)(t)\,dt$.

The function $(\Gamma(b-iax)/\Gamma(b))^c$ is a Schwartz function on $\R$ and in particular integrable, so \eqref{eq:ec(a,b)} follows from \eqref{eq:s(a,b)z}, and $e_c(a,b)$ is $C^\infty$ on $(0,\infty)$.

In this way we have established (i)-(iii) of Theorem~\ref{thm:Urbab}. The proof of the more difficult part (iv) as well as  the proof of Proposition~\ref{thm:negdef}
will be given in Section 3.

 By Riemann-Lebesgue's Lemma we also see that 
$te_c(a,b)(t)$ tends to zero for $t$ tending to zero and to infinity. Much more on the behaviour near 0 and infinity will be given in  Section 2. There we extend the work of \cite{B:L} leading to the asymptotic behaviour of the densities $e_c(a,b)(t)$ for $t\to 0$ and $t\to \infty$.
The behaviour for $t\to\infty$ will lead to a  proof of the S-indeterminacy for $ac>2$ using the Krein criterion.
 
The fact that $\tau_c(a,b)\diamond\tau_d(a,b)=\tau_{c+d}(a,b)$ can be written
\begin{equation}\label{eq:conv}
e_{c+d}(a,b)(t)=\int_0^\infty e_c(a,b)(t/x)e_d(a,b)(x)\,\frac{dx}{x},\quad c,d>0.
\end{equation}
In particular for $c=d=1$ and the explicit formula for $e_1(a,b)$ we get
\begin{eqnarray}\label{eq:e2}
e_2(a,b)(t)&=&\frac{t^{b/a-1}}{[a\Gamma(b)]^2}\int_0^\infty \exp\left(-x^{-1/a}t^{1/a}-x^{1/a}\right)\,\frac{dx}{x}\\
&=&\frac{2t^{b/a-1}}{a\Gamma(b)^2} K_0(2t^{1/(2a)}),\nonumber
\end{eqnarray}
because the Macdonald function $K_0$ is given by
$$
K_0(z)=\frac12\int_0^\infty\exp(-(z/2)^2/y-y)\,\frac{dy}y,
$$ 
 cf. \cite[8.432(7)]{G:R}, \cite[Chap. 10, Sec. 25]{O:M}.

\section{Main results}
Our main results are

\begin{thm}\label{thm:main} For $c>0$ we have
\begin{equation}\label{eq:as1}
e_c(a,b)(t) = \frac{(2\pi)^{(c-1)/2}}{a\sqrt{c}\Gamma(b)^c}\frac{\exp(-ct^{1/(ac)})}{t^{1-(b-1/2+1/(2c))/a}}
\left[1+\mathcal O\left(t^{-1/(ac)}\right)\right],
\quad t\to\infty. 
\end{equation}
\end{thm}

\begin{thm}\label{thm:indet}  
The measure $\tau_c(a,b)\,dt$ is S-indeterminate if and only if $ac>2$.
\end{thm}

\begin{thm}\label{thm:main2} For $c>0$ and $0<t<1$ we have
\begin{equation}\label{eq:asat0}
e_c(a,b)(t)=\frac{t^{b/a-1}}{[a\Gamma(b)]^c}\frac{[\log(1/t)]^{c-1}}{\Gamma(c)} +\mathcal O\left(t^{b/a-1}[\log(1/t)]^{c-2}\right),\quad t\to 0.
\end{equation}
\end{thm}

\begin{rem}\label{thm:remark2} {\rm Formula \eqref{eq:asat0} shows that $e_c(a,b)(t)$ tends to 0 for $t\to 0$ if $b/a>1$, and to infinity if $b/a<1$, independent of $c$. If $b/a=1$ then $e_c(a,b)(t)$ tends to 0 for $c<1$ and to infinity as a power of $\log(1/t)$ when $c>1$.}
\end{rem}

\section{Proofs} 
\medskip

{\it Proof of Proposition~\ref{thm:negdef}:}
From the Weierstrass product for the entire function $1/\Gamma(z)$, we get  the following holomorphic branch  in the cut plane $\mathcal A$, cf. \eqref{eq:cut},
\begin{equation}\label{eq:holbr}
-\log\Gamma(z)=\gamma z + \Log z
+\sum_{k=1}^\infty\left(\Log(1+z/k)-z/k\right),\quad z\in\mathcal A, 
\end{equation}
where $\Log$ denotes the principal logarithm, and $\gamma$ is Euler's constant.

For $n\in\N$ and $z\in\mathcal A$ define
\begin{eqnarray*}
\rho_n(z)&=&\gamma z + \Log z
+\sum_{k=1}^n\left(\Log(1+z/k)-z/k\right),\\ 
R_n(z)&=&\sum_{k=n+1}^\infty\left(\Log(1+z/k)-z/k\right)
\end{eqnarray*}
so $\lim_{n\to\infty} \rho_n(z)=-\log\Gamma(z)$,  uniformly on compact subsets of $\mathcal A$.

Furthermore, we have
$$
\log\Gamma(b)+\rho_n(b)+R_n(b)=0,
$$
and since $\log(1+x)<x$ for $x>0$, we see that $R_n(b)<0$ and hence $\log\Gamma(b)+\rho_n(b)>0$.

We claim that $\log\Gamma(b)+\rho_n(b-iax)$ is a continuous negative definite function, and letting $n\to\infty$ we get the assertion of Proposition~\ref{thm:negdef}. 

To see the claim, we write
\begin{eqnarray*}
&&\log\Gamma(b)+\rho_n(b-iax)=\\
&&\log\Gamma(b) + (b-iax)\left(\gamma-\sum_{k=1}^n\frac{1}{k}\right) +\Log(b-iax)  +\sum_{k=1}^n\Log\left(1+\frac{b-iax}{k}\right)=\\
&&\log\Gamma(b)+\rho_n(b)-iax\left(\gamma-\sum_{k=1}^n\frac{1}{k}\right)+\sum_{k=0}^n
\Log\left(1-i\frac{ax}{b+k}\right),
\end{eqnarray*}
and the assertion follows since $\alpha +i\beta x$ and $\Log(1+i\beta x)$ are
 negative definite functions when $\alpha\ge 0, \beta\in\R$, see \cite{B:F}, \cite{S:S:V}. $\quad\square$ 

\medskip
{\it Proof of Theorem~\ref{thm:main}}:

We modify the proof given in \cite{B:L} and start by applying Cauchy's integral theorem to move the integration in \eqref{eq:ec(a,b)} to a horizontal line
\begin{equation}\label{eq:line}
H_\delta:=\{z=x+i\delta \;:\; x\in\R\},\quad \delta>-b/a.
\end{equation}

\begin{lemma}\label{thm:shift} With $H_\delta$ as in \eqref{eq:line} we have
\begin{equation}\label{eq:def}
e_c(a,b)(t)=\frac{1}{2\pi}\int_{H_\delta} t^{iz-1}[\Gamma(b-iaz)/\Gamma(b)]^c\,dz,\quad t>0.
\end{equation}
\end{lemma}

\begin{proof} For $t,c>0$ fixed, $f(z)=t^{iz-1}[\Gamma(b-iaz)/\Gamma(b)]^c$ is holomorphic in the simply connected domain $\mathbb C\setminus i(-\infty,-b/a]$, so \eqref{eq:def} follows from Cauchy's integral theorem provided the integral
$$
\int_0^{\delta} f(x+iy)\,dy
$$
tends to 0 for $x\to\pm\infty$. We have
$$
|f(x+iy)|=t^{-y-1}|\Gamma(b+y-iax)/\Gamma(b)|^c
$$
and since
$$
 |\Gamma(u+iv)|\sim \sqrt{2\pi}e^{-\pi/2|v|}|v|^{u-1/2},\quad |v|\to\infty, \;\mbox{uniformly for bounded real $u$},
$$
cf. \cite[p.141, Eq. 5.11.9]{NISTA:R}, \cite[8.328(1)]{G:R}, the result follows.
\end{proof}

\medskip
In the following we will use  Lemma~\ref{thm:shift} with the line of integration $H_\delta$, where $
\delta=(t^{1/(ac)}-b)/a$.
Therefore, 
$$
e_c(a,b)(t)=t^{(b-t^{1/(ac)})/a-1}\frac{1}{2\pi}\Int t^{ix}[\Gamma(t^{1/(ac)}-iax)/\Gamma(b)]^c\,dx,
$$
and after the change of variable $x=a^{-1}t^{1/(ac)}u$ and putting $A:=(1/c+b-a)/a$
\begin{equation}\label{eq:def2}
e_c(a,b)(t)=t^{A-a^{-1}t^{1/(ac)}}\frac{1}{2\pi a}\Int t^{iua^{-1}t^{1/(ac)}}[\Gamma(t^{1/(ac)}(1-iu))/\Gamma(b)]^c\,du.
\end{equation}

Binet's formula for $\Gamma$ is (\cite[8.341(1)]{G:R})
\begin{equation}\label{eq:Binet}
\Gamma(z)=\sqrt{2\pi}z^{z-\frac12}e^{-z+\mu(z)},\quad \textrm{Re\,}(z)>0,
\end{equation}
where
\begin{equation}\label{eq:Binet1}
\mu(z)=\int_0^\infty\left(\frac{1}{2}-\frac{1}{t}+\frac{1}{e^t-1}\right)\frac{e^{-zt}}{t}\,dt,\quad \textrm{Re\,}(z)>0.
\end{equation}
Notice that $\mu(z)$ is the Laplace transform of a positive function, so we
have the estimates for $z=r+is, r>0$
\begin{equation}\label{eq:Binet2}
|\mu(z)|\le \mu(r)\le \frac{1}{12r},
\end{equation}
where the last inequality is a classical version of Stirling's formula, thus showing that the estimate is uniform in $s\in\mathbb R$.

Inserting this in \eqref{eq:def2}, we get after some simplification
\begin{equation}\label{eq:def3}
e_c(a,b)(t)=\frac{(2\pi)^{c/2-1}}{a\Gamma(b)^c}t^{A-1/(2a)}e^{-ct^{1/(ac)}}\Int e^{ct^{1/(ac)}f(u)}g_c(u)M(u,t)\,du,
\end{equation}
where
\begin{equation}\label{eq:hp1}
f(u):=iu+(1-iu)\Log(1-iu),\quad g_c(u):=(1-iu)^{-c/2}
\end{equation}
and
\begin{equation}\label{eq:hp2}
M(u,t):=\exp[c\mu(t^{1/(ac)}(1-iu))].
\end{equation}
From \eqref{eq:Binet2} we get $M(u,t)=1+\mathcal O(t^{-1/(ac)})$ for
$t\to\infty$, uniformly in $u$.
We shall therefore consider the behaviour for large $x$ of
\begin{equation}\label{eq:final}
\Int  e^{xf(u)}g_c(u)\,du, \quad x=ct^{1/(ac)}.
\end{equation}

This is the same integral which was treated in \cite[Eq.(28)]{B:L} leading to
$$
\Int  e^{xf(u)}g_c(u)\,du=(2\pi/x)^{1/2}[1+\mathcal O(x^{-1})]
$$
by methods from \cite{L:P:P}.

For $x=ct^{1/(ac)}$ we find 
$$
\int_{-\infty}^\infty e^{ct^{1/(ac)}f(u)}g_c(u)du=\frac{\sqrt{2\pi}}{\sqrt{c}t^{1/(2ac)}}[1+{\mathcal O}(t^{-1/(ac)})],
$$
hence
$$
e_c(a,b)(t)=\frac{(2\pi)^{(c-1)/2}}{a\sqrt{c}\Gamma(b)^c}\frac{e^{-ct^{1/(ac)}}}{t^{1-(b-1/2-1/(2c))/a}}[1+{\mathcal O}(t^{-1/(ac)})].
$$
$\quad\square$

\medskip
{\it Proof of Theorem~\ref{thm:indet}.} 

We first prove that $(s_n(a,b)^c)$ is S-determinate for $ac\le 2$ by Carleman's criterion, cf. \cite[p. 20]{S:T}. In fact, from Stirling's formula we have
$$
s_n(a,b)^{c/(2n)}=(\Gamma(na+b)/\Gamma(b))^{c/(2n)}\sim (na/e)^{ac/2},\quad n\to\infty,
$$ 
so $\sum s_n(a,b)^{-c/(2n)}=\infty$ if and only if $ac\le 2$.

Since Carleman's criterion is only a sufficient condition for S-determinacy, we need to prove that 
$e_c(a,b)$ is S-indeterminate for $ac>2$.
We apply the Krein criterion for S-indeterminacy of probability densities concentrated on the half-line, using a version due to H. L. Pedersen given in \cite[Theorem 4]{H}. 
It states that if 
\begin{equation}\label{eq:as3}
\int_K^\infty \frac{\log e_c(a,b)(t^2)\,dt}{1+t^2}>-\infty
\end{equation}
for some $K\ge 0$, then $\tau_c(a,b)=e_c(a,b)(t)\,dt$ is S-indeterminate.
This version of the Krein criterion is a simplification of a stronger version given in \cite{P}.
 We shall see that
\eqref{eq:as3} holds for $ac>2$.

From Theorem~\ref{thm:main} we see that
\eqref{eq:as3} holds for sufficiently large $K>0$ if and only if
$$
\int_K^\infty \frac{-ct^{2/(ac)}}{1+t^2}\,dt>-\infty,
$$
and the latter holds precisely for $ac>2$. This shows that $\tau_c(a,b)$ is S-indeterminate for $ac>2$.
$\quad\square$

\medskip
{\it Proof of Theorem~\ref{thm:main2}.}
\medskip

The proof uses the same ideas as in \cite{B:L}, but since the proof is quite technical, we give the full proof with the necessary modifications.
Since we are studying the behaviour for $t\to 0$, we assume that
$0<t<1$ so that $\La:=\log(1/t)>0$.

We will need integration along vertical lines
\begin{equation}\label{eq:vert1}
V_\a:=\{\a+iy\mid y=-\infty\ldots\infty\},\quad \a\in \R,
\end{equation}
and we can therefore  express \eqref{eq:ec(a,b)} as
\begin{equation}\label{eq:fund1}
e_c(a,b)(t)=\frac{t^{b/a-1}}{2\pi ia\Gamma(b)^c }\int_{V_{-b}} t^{z/a}\Gamma(-z)^cdz.
\end{equation}
By the functional equation for $\Gamma$ we get
\begin{equation}\label{eq:cont1}
e_c(a,b)(t)=\frac{t^{b/a-1}}{2\pi i a\Gamma(b)^c} \int_{V_{-b}} g(z)\varphi(z)dz,
\end{equation}
where we have defined
$$
\varphi(z):=t^{z/a}\Gamma(1-z)^{c},\quad g(z):=(-z)^{-c}=\exp(-c\Log(-z)).
$$
Note that $\varphi$ is holomorphic in $\mathbb C\setminus[1,\infty)$, while $g$
is holomorphic in $\mathbb C\setminus [0,\infty)$. 

For $x>0$ we define
$$
g_{\pm}(x):=\lim_{\eps\to 0+}g(x \pm i\eps)=x^{-c}e^{\pm i\pi c}.
$$

\medskip
{\bf Case 1.} Assume $0<c<1$.

We fix $0<s<1$, choose $0<\eps<\min(s,b)$  and integrate $g(z)\varphi(z)$ over the contour
$\mathcal C$
$$
\{-b+iy \mid y=\infty\ldots 0\}\cup [-b,-\eps]\cup \{\eps e^{i\theta}\mid\theta=\pi\ldots 0\}\cup[\eps,s]\cup\{s+iy\mid y=0\ldots \infty\}
$$ 
and get 0 by the integral theorem of Cauchy. On the interval $[\eps,s]$ we use  $g=g_+$.

Similarly we get 0 by integrating  $g(z)\varphi(z)$ over the complex conjugate contour $\overline{\mathcal C}$,
and now we use  $g=g_{-}$ on the interval $[\eps,s]$.

Subtracting the second contour integral from the first leads to
$$ 
\int_{V_s}-\int_{V_{-b}}-\int_{|z|=\eps} g(z)\varphi(z)\,dz + \int_\eps^s \varphi(x)
(g_{+}(x)-g_{-}(x))\,dx=0,
$$
where the integral over the circle is with positive orientation. Note
that the two integrals over $[-b,-\eps]$ cancel. Using that $0<c<1$ it is easy to see that the just mentioned integral converges to $0$ for $\eps\to 0$, and we finally get for $\eps\to 0$ 
 $$
e_{c}(a,b)(t)=\frac{t^{b/a-1}}{2\pi i a\Gamma(b)^c}\int_{V_s} g(z)\varphi(z)\,dz +
\frac{t^{b/a-1}\sin(\pi c)}{\pi a\Gamma(b)^c}\int_0^s  x^{-c}\varphi(x)\,dx:=I_1+I_2.
$$
We claim that  $I_1$ is $o(t^{(s+b)/a-1})$ for $t\to 0$. To
see this we insert the parametrization of $V_s$ and get
\begin{eqnarray*}  
I_1&=&\frac{t^{b/a-1}}{2\pi a\Gamma(b)^c}\Int (-s-iy)^{-c}t^{(s+iy)/a}\Gamma(1-s-iy)^{c}\,dy\\
&=&\frac{t^{(s+b)/a-1}}{2\pi a\Gamma(b)^c}\Int e^{-iy\Lambda /a}(-s-iy)^{-c}\Gamma(1-s-iy)^c\,dy,
\end{eqnarray*}
and the integral is $o(1)$ for $t\to 0$ by Riemann-Lebesgue's Lemma because $\Lambda:=\log(1/t)\to\infty$.

The substitution $u=x\La$ in the integral in the term $I_2$ leads to
\begin{equation}\label{eq:I2}  
I_2=\frac{t^{b/a-1}\sin(\pi c)}{\pi a\Gamma(b)^c}\La^{c-1}
\int_0^{s\La} u^{-c}e^{-u/a}\Gamma(1-u/\La)^c\,du.
\end{equation}
We split the integral in \eqref{eq:I2} as
\begin{equation}\label{eq:I2a}
\int_0^{s\La} u^{-c}e^{-u/a}\left[\Gamma(1-u/\La)^c -1\right]\,du + \int_0^\infty u^{-c}e^{-u/a}\,du - \int_{s\La}^\infty u^{-c}e^{-u}\,du.
\end{equation}
Calling the three terms $J_1, J_2, J_3$
we have $J_2=a^{1-c}\Gamma(1-c)$ and
$$
J_3=-a^{1-c}\Gamma(1-c,s\La/a),
$$
where $\Gamma(\a,x)$ is the incomplete Gamma function with the asymptotics 
$$
\Gamma(\a,x)=\int_x^\infty u^{\a-1}e^{-u}\,du\sim x^{\a-1}e^{-x},\quad x\to\infty,
$$
cf. \cite[8.357]{G:R}, hence
$$
J_3=\mathcal O(t^{s/a}\La^{-c}),\quad t\to 0.
$$
Using the Digamma function $\Psi=\Gamma'/\Gamma$, we get by the mean-value theorem 
$$
\Gamma(1-u/\La)^c -1=-\frac{u}{\La}c\Gamma(1-\theta
u/\La)^c\Psi(1-\theta u/\La)
$$
for some $0<\theta<1$, but this implies that
$$
|\Gamma(1-u/\La)^c-1|\le \frac{cu}{\La}M(s),\quad
0<u<s\La,
$$
where
$$
M(s):=\max\{\Gamma(x)^c|\Psi(x)|\mid 1-s\le x\le 1\}.
$$
so $J_1=\mathcal O(\La^{-1})$ for $t\to 0$.

This gives
\begin{eqnarray*}
I_2&=&\frac{t^{b/a-1}\sin(\pi c)}{\pi a\Gamma(b)^c}\La^{c-1}\left(\mathcal O(\La^{-1})+a^{1-c}\Gamma(1-c)+\mathcal O(t^{s/a}\La^{-c})\right)\\
&=&\frac{t^{b/a-1}\La^{c-1}}{(a\Gamma(b))^c\Gamma(c)}+\mathcal O(t^{b/a-1}\La^{c-2}),
\end{eqnarray*}
where we have used Euler's reflection formula for $\Gamma$.
Since finally
$$
I_1=o(t^{(s+b)/a-1})=\mathcal O(t^{b/a-1}\La^{c-2}),
$$
we see that \eqref{eq:asat0} holds.

\bigskip
{\bf Case 2.} Assume $1<c<2$.

The Gamma function decays so rapidly on vertical lines $z=\a+iy, y\to\pm\infty$, that we can integrate by parts in
\eqref{eq:cont1} to get
\begin{equation}\label{eq:cont3}
e_{c}(a,b)(t)=-\frac{t^{b/a-1}}{2\pi i a\Gamma(b)^c}\int_{V_{-1}} \frac{(-z)^{-(c-1)}}{c-1}\frac{d}{dz}\left(t^{z/a}\Gamma(1-z)^c\right)\,dz.
\end{equation}
Defining
$$
\varphi_1(z):=\frac{d}{dz}\left(t^{z/a}\Gamma(1-z)^c\right)=t^{z/a}\Gamma(1-z)^c\left((1/a)\log t-c\Psi(1-z)\right),
$$
and using the same contour technique as in Case 1 to the integral in \eqref{eq:cont3}, where now $0<c-1<1$, we get for $0<s<1$ fixed
$$
e_c(a,b)(t)=-\frac{t^{b/a-1}}{a \Gamma(b)^c}\left(\tilde{I}_1+\tilde{I}_2\right),
$$
where
\begin{eqnarray*}
\tilde{I}_1&=&\frac{1}{2\pi i (c-1)}\int_{V_s}(-z)^{-(c-1)}\varphi_1(z)\,dz,\\ 
\tilde{I}_2&=&\frac{\sin(\pi(c-1))}{\pi(c-1)}\int_0^s x^{-(c-1)}\varphi_1(x)\,dx.
\end{eqnarray*}
We have $\tilde{I}_1=o(t^{s/a}\La)$ for $t\to 0$ by Riemann-Lebesgue's Lemma, and the substitution
$u=x\La$ in the second integral leads to
\begin{eqnarray*}
\lefteqn{\int_0^s
x^{-(c-1)}\varphi_1(x)\,dx}\\
&=&\La^{c-2}\int_0^{s\La}u^{-(c-1)}\varphi_1(u/\La)\,du\\
&=&-(1/a)\La^{c-1}\left(\int_0^{s\La} u^{-(c-1)}e^{-u/a}\,du+\int_0^{s\La} u^{-(c-1)}e^{-u/a}\left(\Gamma(1-u/\La)^c-1\right)\,du\right)\\
&-&
c\La^{c-2}\int_0^{s\La}u^{-(c-1)}e^{-u/a}\Gamma(1-u/\La)^c\Psi(1-u/\La)\,du\\
&=&-a^{1-c}\La^{c-1}\Gamma(2-c)+\mathcal O(\La^{c-2}).
\end{eqnarray*}
Using that
$$
 \frac{\sin(\pi(c-1))}{(c-1)\pi}\left(-a^{1-c}\La^{c-1}\Gamma(2-c)\right)=-a^{1-c}\frac{\La^{c-1}}{\Gamma(c)}
$$
by Euler's reflection formula, we see that \eqref{eq:asat0} holds.

\bigskip
{\bf Case 3.} Assume $c>2$.

We perform the change of variable $w=(1/a)\La z$ in \eqref{eq:cont1} and assume that $\La >a$. This gives
$$
e_c(a,b)(t)=\frac{t^{b/a-1}\La^{c-1}}{[a\Gamma(b)]^c}\frac{1}{2\pi i}\int_{V_{-(b/a)\La}}(-w)^{-c} e^{-w}\Gamma(1-aw/\La)^c\,dw.
$$
  Using Cauchy's integral theorem, we can shift the contour $V_{-(b/a)\La}$
  to $V_{-1}$  as the integrand is holomorphic in the vertical strip between both paths and exponentially small at both extremes of that vertical strip. 
For the holomorphic function $h(z)=\Gamma(1-z)^c$ in the  domain $G=\mathbb C\setminus[1,\infty)$, which is star-shaped
with respect to $0$, we have
$$
h(z)=h(0)+z\int_0^1 h'(uz)\,du,\quad z\in G,
$$ 
hence
\begin{equation}\label{eq:taylor}
\Gamma(1-aw/\La)^c=1-\frac{caw}{\La}\int_0^1 \Gamma(1-uaw/\La)^c\Psi(1-uaw/\La)\,du.
\end{equation}
 Defining 
$$
R(w)=\int_0^1 \Gamma(1-uaw/\La)^c\Psi(1-uaw/\La)\,du,
$$
we get
\begin{eqnarray*}
\lefteqn{\frac{1}{2\pi i}\int_{V_{-1}}(-w)^{-c}e^{-w}\Gamma(1-aw/\La)^c\,dw =}\\
&& \frac{1}{2\pi i}\int_{V_{-1}} (-w)^{-c} e^{-w}dw +
\frac{ac/\Lambda}{2\pi i}\int_{V_{-1}}(-w)^{1-c} e^{-w}R(w)dw.
\end{eqnarray*}
For any $w\in V_{-1}, 0\le u\le 1$ and for $\La\ge a$ we have that
 $1-uaw/\La$ belongs to  the closed vertical strip
 located between the vertical lines $V_1$ and $V_2$. Because
$\Gamma(z)^c\Psi(z)$ is  continuous and bounded in this strip, $R(w)$ is bounded for $w\in V_{-1}$ by a constant independent of $\La\ge a$. Furthermore,
 $(-w)^{1-c}e^{-w}$ is integrable over $V_{-1}$
because $c>2$.

On the other hand, in the  integral
$$
\frac{1}{2\pi i}\int_{V_{-1}} (-w)^{-c} e^{-w}dw
$$
the contour $V_{-1}$ may be deformed to a Hankel contour 
$$
\mathcal H:=\{x-i \mid x=\infty\ldots  0\}\cup  \{e^{i\theta}\mid \theta=-\pi/2 \ldots -3\pi/2\}
\cup\{x+i\mid x=0\ldots \infty\}
$$
surrounding $[0,\infty)$, and the integral over $\mathcal H$ is Hankel's integral representation of the inverse of the Gamma function:
$$
\frac{1}{2\pi i}\int_{\mathcal H} (-w)^{-c} e^{-w}dw=\frac{1}{\Gamma(c)}.
$$
Therefore,  when we join everything, we obtain that for $c>2$:
$$
e_c(a,b)(t)=\frac{t^{b/a-1}}{[a\Gamma(b)]^c}\frac{[\log(1/t)]^{c-1}}{\Gamma(c)}+\mathcal O\left(t^{b/a-1}[\log(1/t)]^{c-2}\right), \quad t\to 0.
$$

\medskip
{\bf Case 4.} $c=1,c=2$.

These cases are easy since $e_1(a,b)(t)$ is explicitly given by \eqref{eq:ta} and $e_2(a,b)(t)$ by
\eqref{eq:e2}. The asymptotics of $K_0$ is known:
$$
K_0(t)=\log(2/t)+\mathcal O(1),\quad t\to 0.
$$ 
$\quad\square$

\begin{rem} {\rm The behaviour of $e_c(a,b)(t)$ for $t\to 0$ can be
    obtained from \eqref{eq:fund1} using the residue theorem when $c$
    is a natural number. In fact, in this case $\Gamma(-z)^c$ has a
    pole of order $c$ at $z=0$, and a shift of the contour $V_{-1}$ to
    $V_s$, where $0<s<1$, has to be compensated by a residue, which
    will give the behaviour for $t\to 0$.

When $c$ is a natural number one can actually express $e_c(a,b)(t)$
in terms of Meijer's G-function:
$$
e_c(a,b)(t)=\frac{t^{b/a-1}}{a\Gamma(b)^c} G^{c,0}_{0,c}\left(t^{1/a} \mid \begin{array}{lcl}
- & \cdots & -\\ 0 & \cdots & 0 \end{array}  \right),
$$
cf. Section 9.3 in \cite{G:R}.
}
\end{rem}

\section{A one parameter extension of the Gumbel distributions}

The group isomorphism $x=\log(1/t)$ of the multiplicative group $(0,\infty)$ onto the additive group $\R$ of real numbers transforms the convolution semigroup $(\tau_c(a,b))_{c>0}$  into an ordinary
convolution semigroup $(G_c(a,b))_{c>0}$ of probability measures on $\R$ with densities given by
\begin{equation}\label{eq:Guc}
g_c(a,b)(x)=e^{-x}e_c(a,b)(e^{-x}),\quad x\in\R,
\end{equation}
and $a,b,c>0$ are arbitrary. For $c=1$ we have
\begin{equation}\label{eq:Gu1}
g_1(a,b)(x)=\frac1{a\Gamma(b)}\exp\left(-bx/a-e^{-x/a}\right),\quad x\in\R.
\end{equation}
This density is infinitely divisible and the uniquely determined convolution roots are given by \eqref{eq:Guc}.

The special density $g_1(a,1)(x)$ is the Gumbel density with scale parameter $a>0$, and the basic case $a=1$ is discussed in \cite{B:L}. From the asymptotic behaviour of $e_c(a,b)$ in Theorems
\ref{thm:main} and \ref{thm:main2} we can obtain the asymptotic behaviour of the convolution roots $g_c(a,b)$:
\begin{equation}\label{eq:g-inf}
g_c(a,b)(x)=\frac{(2\pi)^{(c-1)/2}}{a\sqrt{c}\Gamma(b)^c}\frac{\exp\left(-ce^{-x/(ac)}\right)}{\exp\left(x(b-1/2+1/(2c))/a\right)}
\left[1+\mathcal O\left(\exp(x/(ac)\right)\right]
\end{equation}
for $x\to-\infty$, and
\begin{equation}\label{eq:ginf}
g_c(a,b)(x)=\frac{\exp(-bx/a)x^{c-1}}{[a\Gamma(b)]^c\Gamma(c)} +\mathcal O\left(\exp(-bx/a)x^{c-2}\right),\quad x\to\infty.
\end{equation}

\begin{thm}\label{thm:Gudet}
All densities $g_c(a,b)$ belong to determinate Hamburger moment problems.
\end{thm}

\begin{proof} We first prove that $g_1(a,b)$ is determinate, and for this it suffices to verify that 
the moments
\begin{equation}\label{eq:g1mom}  
s_n=\Int x^n g_1(a,b)(x)\,dx
\end{equation}
verify Carleman's condition  $\sum_{n=0}^\infty s_{2n}^{-1/(2n)}=\infty$,
cf. \cite[p. 19]{S:T}. From \eqref{eq:g1mom} we get
\begin{eqnarray*}
s_{2n}&=&\frac{1}{a\Gamma(b)}\int_0^\infty (\log t)^{2n}t^{b/a-1}\exp(-t^{1/a})\,dt=
\frac{1}{\Gamma(b)}\int_0^\infty (a\log s)^{2n}s^{b-1}e^{-s}\,ds\\
&<&\frac{a^{2n}}{\Gamma(b)}\left(\int_0^1 (\log s)^{2n}s^{b-1}\,ds +\int_1^\infty s^{2n+b-1}e^{-s}\,ds\right).
\end{eqnarray*}
By integrations by parts we see that
$$
\int_0^1 (\log s)^{2n}s^{b-1}\,ds=\frac{(2n)!}{b^{2n+1}},
$$
and
$$
\int_1^\infty s^{2n+b-1}e^{-s}\,ds<\Gamma(2n+b),
$$
hence
$$
s_{2n}^{1/(2n)}<\frac{a}{\Gamma(b)^{1/(2n)}}\left[\left(\frac{(2n)!}{b^{2n+1}}\right)^{1/(2n)} + \Gamma(2n+b)^{1/2n}\right],
$$
and the Carleman condition follows from Stirling's formula, which shows that the right-hand side is bounded by $Kn$ for sufficiently large $K>0$. We next use Corollary 3.3 in \cite{B0} to infer that the Carleman condition also holds for all convolution roots $g_c(a,b)$.
\end{proof}

Concerning the moments
\begin{equation}\label{eq:gcmom}
s_n(c)=\Int x^ng_c(a,b)(x)\,dx, n\in\N_0
\end{equation}
of the convolution roots we have the following result:

\begin{thm}\label{thm:momroots} The moments $s_n(c)$ of \eqref{eq:gcmom} is a polynomial
\begin{equation}\label{eq:gcmom1}
s_n(c)=\sum_{k=1}^n a_{n,k}c^k,\quad n\ge 1,
\end{equation}
of degree at most $n$ in the variable $c$. The coefficients $a_{n,k}$ are given below.
\end{thm} 

\begin{proof} From \eqref{eq:s(a,b)z}  we get
$$
\Int e^{-ixy}\,dG_c(a,b)(x)=\int_0^\infty t^{iy} e_c(a,b)(t)\,dt=[\Gamma(b+iay)/\Gamma(b)]^c,
$$
which shows that the negative definite function $\rho$ corresponding to the convolution semigroup 
$(G_c(a,b))_{c>0}$ is
$$
\rho(y)=\log\Gamma(b)-\log\Gamma(b+iay), \quad y\in\R.
$$
The derivatives of $\rho$ can be expressed in terms of the Digamma function $\Psi$, namely
$$
\rho^{(n+1)}(y)=-(ia)^{n+1}\Psi^{(n)}(b+iay), \quad n\in\N_0,
$$
so if  for $n\in\N_0$ we define (cf.  \cite[Eq. (2.7)]{B})
$$
\sigma_n:=-i^{n+1}\rho^{(n+1)}(0)=(-a)^{n+1}\Psi^{(n)}(b),
$$
we find
\begin{eqnarray*}
\sigma_0 &=& a\gamma +\frac{a}{b}-ab\sum_{k=1}^\infty \frac{1}{k(b+k)},\\
\sigma_n &=& a^{n+1}n!\sum_{k=0}^\infty\frac{1}{(b+k)^{n+1}}=a^{n+1}n!\zeta(n+1,b),\quad n\in\N,
\end{eqnarray*}
where $\zeta(z,q)$ is Hurwitz' Zeta function, cf. \cite[9.521]{G:R}.

According to \cite{B} we have $s_1(c)=\sigma_0 c, s_2(c)=\sigma_1 c+\sigma_0^2 c^2$ and in general $s_n(c)$ is given by \eqref{eq:gcmom1} where the coefficients $a_{n,k}$ are given by the recursion
$$
a_{n+1,k+1}=\sum_{j=k}^n a_{j,k}\binom{n}{j}\sigma_{n-j},\quad n\ge k\ge 0.
$$
It is easy to see that
$$
a_{n,1}=\sigma_{n-1},\quad a_{n,n-1}=\binom{n}{2}\sigma_0^{n-2}\sigma_1,\quad a_{n,n}=\sigma_0^n.
$$
\end{proof}

\medskip
C. Berg, Department of Mathematical Sciences, University of
Copenhagen, Universitetsparken 5,
2100 Copenhagen {\O}, Denmark

email: berg@math.ku.dk

\end{document}